\documentclass[10pt]{article}

\usepackage{amssymb,amsmath,amsthm, amsfonts,graphicx}
\usepackage{color}

\newtheorem{defi}{Definition}
\newtheorem{proposition}{Proposition}

\title{\LARGE \bf Distributionally Robust  Games: \\   $f$-Divergence and Learning}
\author{Dario Bauso$^{1},$ Jian Gao$^{2}$ and Hamidou Tembine$^{2}$
\thanks{$^{1}$ Department of Automatic Control and Systems Engineering, The University of Sheffield, Mappin Street, Sheffield, S1 3JD, United Kingdom {\tt\small d.bauso@sheffield.ac.uk}.}
\thanks{$^{2}$ Learning \& Game Theory Laboratory, 
New York University Abu Dhabi {\tt\small tembine@nyu.edu}.}
}

\begin{document}


\maketitle

\begin{abstract}
 In this paper we introduce the novel framework of distributionally robust games. These are multi-player games where each player models the state of nature using a worst-case distribution, also called adversarial distribution. Thus each player's payoff depends on the other players' decisions and on the decision of a virtual player (nature) who selects an adversarial distribution of scenarios. This paper provides three main contributions. Firstly, the distributionally robust game is formulated using the statistical notions of $f$-divergence between two distributions, here represented by the adversarial distribution, and the exact distribution.  Secondly,  the complexity of the problem is significantly reduced by means of triality theory. Thirdly, stochastic  Bregman  learning algorithms are proposed to speedup the computation of robust equilibria. Finally, the theoretical findings are illustrated in a convex setting and its limitations are tested with a  non-convex non-concave function.
\end{abstract}


\newpage

\tableofcontents
\newpage

\section{Introduction}

Games with payoff uncertainty refer to games where the outcome of a play is not known with certainty by the players. Such games are also called incomplete information games and can be formalized in different ways. Distribution-free models of incomplete information games, both with and without private information are examined in \cite{free,free2}. There the players use a robust optimization approach to contend with the worst-case scenario payoff. The distribution free approach relaxes the well-known Bayesian game model approach by Harsanyi. The limitation of the distribution-free model is that the uncertainty set has to be carefully designed and in most cases such approach leads to too conservative and unrealistic scenarios.     

Strategic learning has  proven to be a powerful approach in stochastic games. In particular, its algorithmic nature is well suited to accommodate parallel and distributed information exchange and processing as well as hardware realizability. However, almost all  existing learning approaches work well only for specific classes of games such as concave-convex zero-sum games, convex potential games, and some S-modular game problems with unimodal objective functions. For more general classes of games, convergence of strategic learning dynamics is still an open issue.  
 In addition, learning algorithms for finding fixed points or equilibria for general classes of games still present several challenges. For polymatrix games with finite action spaces, there have been great progress including Cournot adjustment, Brown-von Neumann-Nash dynamics, reinforcement learning, combined learning  (see \cite{book} and the references  therein). For continuous action spaces, however, only few handful works are available. Evolutionary dynamics and revision protocols based actions has been proposed \cite{conti4,conti1,conti5}.  
The exploration of continuous action takes too much time if the dynamics is based on individual action or measurable subsets \cite{conti6,conti3,conti2}. Moreover, the convergence time of these existing strategic learning algorithms (when a convergence to a point or  a limit cycle occurs)  are unacceptably high even for potential games and it often requires strong assumptions such as bounded densities.  
The above mentioned prior works do not consider  robust games setting. In \cite{free,free2} a robust game framework is presented. The authors defined a distributed-free approach (by considering  worst-case performance). However the choice of the uncertainty set remains an important part of the robust game modelling.
In this work we are interested in  learning in distributionally robust games under $f$-divergence.

\subsection{Contribution} We make several contributions in this paper. We introduce for the first time a novel game model, called distributionally robust game. This game provides a new and original way of addressing game scenarios with incomplete information.  For this game, we provide a rigorous definition of distributionally robust equilibrium. Distributionally robust games accommodate both finite and continuous action spaces. 
 The relevance in formulating such a new game is in that it relaxes
 the assumptions of Harsanyi's Bayesian games. Distributionally robust games differ from the distribution-free framework of Aghassi \& Bertsimas in  \cite{free,free2} in that in the distribution-free approach  the interval (or generally the uncertainty set) needs to be known (learnable) by the decision-maker. 
In contrast to this, in the distributionally robust approach any alternative distribution within a divergence ball can be tested. 

As second contribution, we use a \emph{triality theory},  which  reduces considerably the curse of dimensionality of the problem. We prove the existence of equilibria in any such robust finite game under suitable conditions.

As third contribution, we provide a computational method based on Bregman flow for approximately computing equilibria. Such a computational method allows us to test the implementability of the approach on numerical examples.  We introduce a class of distributionally robust games with continuous action spaces, for which a subset of equilibria can be computed using the Bregman algorithm.  We show that the resulting iterative dynamics, which we call Bregman dynamics, is characterized by  double exponential decay and convergence to distributionally robust equilibria.

\subsection{Structure}
 The rest of the paper is structured as follows. In Section \ref{sec:two} we introduce distributionally robust games. Section \ref{sec:three} presents a  learning algorithm for robust equilibria.  Section \ref{sec:four}  focuses on stochastic Bregman learning.
 Section \ref{sec:five} provides numerical results. Discussions on the finite action spaces are presented in Section \ref{sec:six}.
 Section \ref{sec:last} concludes the paper.
   
\section{Distributionally Robust Games} \label{sec:two}
In this section, we first introduce the distribution uncertainty set, and then formulate the distributionally robust game problem. Then, we define distributionally robust equilibrium, discuss triality theory and apply such theory to reduce the model via Lagrangian relaxation.

\subsection{Distribution Uncertainty Set}
Let  $(\Omega, \mathcal{F},m)$ be a probability space. Here $m$  is a probability measure defined on  $(\Omega, \mathcal{F}).$ 
The distribution $m$ of the state $\omega$ is used to capture the probability of the different scenarios and of the corresponding performance function obtained under each of such scenarios for fixed action profile.  We assume that the exact distribution of the state
 is not available in general.  Therefore we propose an uncertainty/constraint set among all the possible  distributions with a divergence bounded by above by a scalar value $\rho.$
Such a constraint takes the form
$$
B_{\rho}(m)=\{ \tilde{m} \ | \ \int_{\Omega} d\tilde{m}=\tilde{m}(\Omega)=1,\ D_{f}(\tilde{m}\ || \ m)\leq \rho\},
$$
 where $D_{f}$ is the so-called $f$-divergence from the probability measure $m$ to $\tilde{m}$ defined as 
$$D_{f}(\tilde{m} \ || \ m)=\int_{\omega \in\Omega} f\left(\frac{d\tilde{m}}{dm}  \right) dm -f(1).
$$

Recall that for a convex (and proper)  function $f$ the  Legendre-Fenchel duality holds: $(f^*)^*=f$ where 
\begin{eqnarray}
f^*(\xi) 
= \sup_{x}[\langle x,\xi\rangle-f(x)]
=-\inf_{x}[ f(x)-\langle x,\xi\rangle].
\end{eqnarray}
\subsection{Problem Formulation}
Each player $j$ chooses $a_j\in \mathcal{A}_j$ to optimize  the worst loss performance functional $\mathbb{E}_{\tilde{m}} l_j(a,\omega) $ subject to the constraint that the divergence $D_f(\tilde{m}\parallel m )\leq \rho$. This means that the worst loss performance is obtained under the assumption that a virtual player (nature) acts as a discriminator/attacker who modifies  the distribution $m$  into  $\tilde{m}$ with an effort  capacity that should not exceed $\rho>0.$    The robust stochastic optimization of player $j$ given $(a_{j'})_{j'\neq j}, m, \rho$ is given by
\begin{equation}\label{eq00} (P_j)\left\{
\begin{array}{l}
\inf_{a_j\in \mathcal{A}_j} \ \sup_{\tilde{m}\in B_{\rho}(m)}\ \mathbb{E}_{\tilde{m}} l_j(a,\omega).
\end{array} \right.
\end{equation}

Throughout the paper we assume that the following hold:
The measure  $\tilde{m}$ is  continuous with the respect to $m$ and it is not a given profile, it could be deformed or falsified by the discriminator.
The function $ l_j(., \omega)$ is proper and upper semi-continuous for $m-$almost all $\omega\in\Omega.$ Either  the domain $\mathcal{A}_j$ is a non-empty  compact  set or  $\mathbb{E}_{\tilde{m}} l_j(a,\omega)$ is coercive. 

\begin{defi}[Distributionally Robust Game]
The robust game $\mathcal{G}(m)$ involves
\begin{itemize}
\item the set of players $\mathcal{J}=\{1,2,\ldots, n\}, \ n\geq 2$ 
\item the decision space $\mathcal{A}_j$ of each player $j$, $j\in \mathcal{J}$
\item the uncertainty set $B_{\rho}(m)$ defined on the set of probability distributions  $m$  on $\Omega$ and $\rho>0$
\item  the payoff function $\mathbb{E}_{\tilde{m}} l_j(a,\omega)$ of player $j$, $j\in \mathcal{J}$.  
\end{itemize}
\end{defi}

With the above game in mind, we can introduce the following solution concept. 
\begin{defi}[Distributionally Robust Equilibrium]
Let $a^*_j$ be the configuration of player $j$ and  $a^*_{-j}:=(a^*_k)_{k\neq j}.$
A strategy profile $a^*=(a^*_1,\ldots,a^*_n)$ satisfying 
$$\sup_{\tilde{m}\in B_{\rho}(m)}\mathbb{E}_{\tilde{m}} l_j(a^*,\omega)\leq \sup_{\tilde{m}\in B_{\rho}(m)}\mathbb{E}_{\tilde{m}} l_j(a_j,a^*_{-j},\omega),$$ for every $\ a_j\in\mathcal{A}_j$ and every agent $j,$  is said \emph{distributionally robust  pure Nash equilibrium} of  game $\mathcal{G}(m).$
\end{defi}

\subsection{From Duality to Triality Theory}\label{sec:TT}
We here streamline the basic idea of triality theory. To this purpose, consider uncoupled domains $\mathcal{A}_j,\ j\in \mathcal{J}.$
For a general function $l_2,$ one has $$\sup_{a_2\in \mathcal{A}_2}\inf_{a_1\in \mathcal{A}_1}l_2(a_1,a_2)\leq   \inf_{a_1\in \mathcal{A}_1}\sup_{a_2\in \mathcal{A}_2}l_2(a_1,a_2)$$ and the difference 
$$ \min_{a_1\in \mathcal{A}_1}\max_{a_2\in \mathcal{A}_2}l_2(a_1,a_2)-\max_{a_2\in \mathcal{A}_2}\min_{a_1\in \mathcal{A}_1}l_2(a_1,a_2)$$ is the well-known duality gap.
As it is widely known in duality theory from Sion's Theorem  \cite{sion} (which is an extension of von Neumann minimax Theorem) there is an equality, for example for convex-concave function, and the value is achieved by a saddle point in the case of non-empty convex compact domain.
For a general function $l_3$, $(a_1,a_2, a_3) \mapsto l_3(a_1,a_2, a_3)$ one has 
\begin{eqnarray*}
\inf_{a_3\in \mathcal{A}_3}\sup_{a_2\in \mathcal{A}_2}\inf_{a_1\in \mathcal{A}_1}l_3(.)\leq   
\inf_{a_1\in \mathcal{A}_1,a_3\in \mathcal{A}_3}\sup_{a_2\in \mathcal{A}_2}l_3(.), \\ 
\sup_{a_3\in \mathcal{A}_3}\inf_{a_2\in \mathcal{A}_2}\sup_{a_1\in \mathcal{A}_1}l_3(.)\geq   
\sup_{a_1\in \mathcal{A}_1,a_3\in \mathcal{A}_3}\inf_{a_2\in \mathcal{A}_2}l_3(.).
\end{eqnarray*}

\begin{proposition}[Triality]
Let $ (a_1,a_2, a_3) \mapsto l_3(a_1,a_2, a_3)\in \mathbb{R}$ be a  function $l_3$ defined on $\prod_{i=1}^3 \mathcal{A}_i.$ Then, the following inequalities hold:
\begin{equation} \label{eqtri}
\begin{array}{ll}
\sup_{a_2\in \mathcal{A}_2}\inf_{a_1\in \mathcal{A}_1,a_3\in \mathcal{A}_3}l_3(a_1,a_2, a_3)\leq \\
\inf_{a_3\in \mathcal{A}_3}\sup_{a_2\in \mathcal{A}_2}\inf_{a_1\in \mathcal{A}_1}l_3(a_1,a_2, a_3)\leq   \\
\inf_{a_1\in \mathcal{A}_1,a_3\in \mathcal{A}_3}\sup_{a_2\in \mathcal{A}_2}l_3(a_1,a_2, a_3),
\end{array}
\end{equation}
and similarly
\begin{equation} \label{eqtri2}
\begin{array}{ll}
\sup_{a_1\in \mathcal{A}_1,a_3\in \mathcal{A}_3}\inf_{a_2\in \mathcal{A}_2}l_3(a_1,a_2, a_3)\leq \\
 \sup_{a_3\in \mathcal{A}_3}\inf_{a_2\in \mathcal{A}_2}\sup_{a_1\in \mathcal{A}_1}l_3(a_1,a_2, a_3)\leq \\
\inf_{a_2\in \mathcal{A}_2}\sup_{a_1\in \mathcal{A}_1,a_3\in \mathcal{A}_3}l_3(a_1,a_2, a_3).
\end{array}
\end{equation}
\end{proposition}
\begin{proof}
First we shall prove the $ \sup\inf$ inequality. Define $$g(a_2,a_3)=\inf_{a_1\in \mathcal{A}_1}l_3(a_1,a_2, a_3).$$ Thus, for all $a_2,a_3,\ $ one has 
$ g(a_2,a_3)\leq l_3(a_1,a_2, a_3).$ 
It follows that, for any $ a_1, a_3,$
$$
 \sup_{a_2\in \mathcal{A}_2} g(a_2,a_3)\leq  \sup_{a_2\in \mathcal{A}_2} l_3(a_1,a_2, a_3).\  
$$

Using the definition of $g,$ one obtain  
$$
 \sup_{a_2\in \mathcal{A}_2} \inf_{a_1\in \mathcal{A}_1}l_3(a_1,a_2, a_3)\leq  \sup_{a_2\in \mathcal{A}_2} l_3(a_1,a_2, a_3),\  \forall a_1,a_3.
$$

Taking the infinimum in $a_1$ yields:
\begin{equation}\label{eq:minmax}
 \sup_{a_2\in \mathcal{A}_2} \inf_{a_1\in \mathcal{A}_1}l_3(a_1,a_2, a_3)\leq  \inf_{a_1\in \mathcal{A}_1}\sup_{a_2\in \mathcal{A}_2} l_3(a_1,a_2, a_3),\  \forall a_3
\end{equation}

Now, for variable in $a_3$ we use two operations:
\begin{itemize}\item Taking the infininum in inequality (\ref{eq:minmax}) in $a_3$ yields

$$
 \inf_{a_3\in \mathcal{A}_3}\sup_{a_2\in \mathcal{A}_2} \inf_{a_1\in \mathcal{A}_1}l_3(a_1,a_2, a_3)\leq  $$  $$   \inf_{a_3\in \mathcal{A}_3}\inf_{a_1\in \mathcal{A}_1}\sup_{a_2\in \mathcal{A}_2} l_3(a_1,a_2, a_3)$$ $$= \inf_{(a_1,a_3)\in \mathcal{A}_1\times \mathcal{A}_3 }\sup_{a_2\in \mathcal{A}_2} l_3(a_1,a_2, a_3),
$$
which proves the second part of the inequalities (\ref{eqtri}). The first part of  the inequalities (\ref{eqtri}) follows immediately from (\ref{eq:minmax}).
\item Taking the supremum in inequality (\ref{eq:minmax}) in $a_3$ yields
 $$\sup_{(a_2,a_3)\in \mathcal{A}_2\times \mathcal{A}_3} \inf_{a_1\in \mathcal{A}_1}l_3(a_1,a_2, a_3)$$  $$\leq   \sup_{a_3\in \mathcal{A}_3}\inf_{a_1\in \mathcal{A}_1}\sup_{a_2\in \mathcal{A}_2} l_3(a_1,a_2, a_3),$$ which proves the first part of the  inequalities (\ref{eqtri2}). The second part of the inequalities (\ref{eqtri2}) follows immediately from (\ref{eq:minmax}).
\end{itemize}
This completes the proof.

\end{proof}

We use the above inequalities in the Lagrangean relaxation of the MaxMin Robust Game. 
\subsection{MaxMin Robust Game: Infinite Dimension}
Assume that $a\mapsto \mathbb{E}_{\tilde{m}} l_j(a,\omega)$ is continuous for $m-$almost all $\omega.$ Then, the functional 
$F_j:\ \tilde{m}\mapsto \inf_{a_j}\mathbb{E}_{\tilde{m}} l_j(a,\omega)$ is Gateaux differentiable with derivative 
$$
F_{j,m} (\hat{m})=\inf_{a_j\in \mathcal{A}_j^*(m)}\mathbb{E}_{\hat{m}} l_j(a,\omega),
$$
where $ \mathcal{A}_j^*(m)=\arg\min_{a_j} \mathbb{E}_{m} l_j(a,\omega)$ is the best-response under $m.$ This derivative in the space of square integrable measurable functions under $m$ which is of infinite dimensions, does not facilitate the computation of the robust optimal strategy $a_j^*, \tilde{m}^*.$  
 Below we propose an equivalent problem that reduces considerably the curse of dimensionality of the problem.

\subsection{MaxMin Robust Game: Dimension Reduction}
In order to reduce the curse of dimensionality of the problem we use a triality theory.
The robust best-response problem  of agent $j$  is equivalent to
\begin{equation}\label{eq022} \left\{
\begin{array}{l}
\inf_{a_j} \sup_{L\in L_{\rho}(m)} \mathbb{E}_{m} [l_j L]; 
\end{array} \right.
\end{equation}
where $L(\omega)=\frac{d\tilde{m}}{dm}(\omega)$ is the likelihood and  set $ L_{\rho}(m)$ is 
$$L_{\rho}(m)=\{ L\  | \   \int_{\omega}  f(L(\omega)) d{m} -f(1)\leq \rho,\ \  \int_{\omega} L(\omega)dm=1\}.$$
We introduce the Lagrangian as 
$$\begin{array}{ll}
\tilde{l}_j(a,L,\lambda,\mu)=&  \int_{\omega}  l_j(a,\omega) L(\omega) dm\\ & +\lambda(\rho+f(1)-\int_{\omega}  f(L(\omega)) d{m})\\ &
+\mu(1-\int_{\omega} L(\omega)dm(\omega)),
\end{array}$$ where $\lambda\geq 0$ and $\mu\in \mathbb{R}.$
The problem solved by Player $j$ is
\begin{equation}\label{eq02} (\tilde{P}^*_j)\left\{
\begin{array}{l}
\inf_{a_j} \sup_{L\in L_{\rho}(m)}\inf_{\lambda\geq 0,\mu\in \mathbb{R}} \tilde{l}_j(a,L,\lambda,\mu).
\end{array} \right.
\end{equation}\
A full understanding  of  problem $ (\tilde{P}^*_j)$ requires a triality theory (not a duality theory) whose main principles were streamlined in Section \ref{sec:TT}. The underlying idea is that  one can use a  transformation of the last two terms to derive a finite dimensional optimization problem.
The Lagrangian $\tilde{l}_j$ of agent $j$  is clearly concave in $L$ and convex in $\lambda,\mu$ and is semi-continuous jointly.
By the triality theory above,  $\tilde{l}_j: (a,L,\lambda,\mu) \mapsto \tilde{l}_j(a,L,\lambda,\mu)$  satisfies the $\sup\inf$ inequality and one has the following: 
$$\inf_{a_j} \sup_{L\in L_{\rho}(m)}\inf_{\lambda\geq 0,\mu\in \mathbb{R}} \tilde{l}_j(.)
\leq \inf_{a_j}\inf_{\lambda\geq 0,\mu\in \mathbb{R}}  \sup_{L\in L_{\rho}(m)}\tilde{l}_j(.).$$
In this case there is no  gap in the second part of the optimization and the following equality holds:
$$\inf_{a_j} \sup_{L\in L_{\rho}(m)}\inf_{\lambda\geq 0,\mu\in \mathbb{R}} \tilde{l}_j(.) 
=\inf_{a_j}\inf_{\lambda\geq 0,\mu\in \mathbb{R}}  \sup_{L\in L_{\rho}(m)}\tilde{l}_j(.).$$
The latter problem can be rewritten as
\begin{equation}\label{eq04} (\tilde{P}^*_j)\left\{
\begin{array}{l}
\inf_{a_j\in \mathcal{A}_j,\lambda\geq 0,\mu\in \mathbb{R}}[\sup_{L\in L_{\rho}(m)} \tilde{l}_j(a,L,\lambda,\mu)].
\end{array} \right.
\end{equation}\
The Lagrangian function takes the form as 
$\tilde{l}_j=  \lambda(\rho+f(1))+\mu +\int \{ L [ {l}_j-\mu] -\lambda f(L)\} dm.$
It follows that
\begin{eqnarray}
\sup_{L\in L_{\rho}(m)} \tilde{l}_j(a,L,\lambda,\mu)=\nonumber \\   \lambda(\rho+f(1))+\mu +\sup_{L}\int \{ L [ {l}_j-\mu] - \lambda f(L)\} dm.
\end{eqnarray}
Introducing the Fenchel-Legendre transform on $L$ and exchanging $\sup$ and $\int,$ one gets
$$
\sup_{L\in L_{\rho}(m)} \tilde{l}_j(.)
= \lambda(\rho+f(1))+\mu +\int \lambda f^*(\frac{ {l}_j-\mu}{\lambda}) dm.
$$
Since $\mathcal{A}_j\times\mathbb{R}_+\times \mathbb{R}$ is a subset of a finite dimensional vector space, then it follows that the robust best-response problem of agent $j$ is 
equivalent to the finite dimensional stochastic optimization problem:
\begin{equation}\label{eq05} ({P}^*_j)\left\{
\begin{array}{l}
\inf_{a_j\in \mathcal{A}_j,\lambda\geq 0,\mu\in \mathbb{R}}l_j^*(a,\lambda,\mu,m)
\\ l_j^*(a,\lambda,\mu,m)=\\
 \lambda(\rho+f(1))+\mu +\int \lambda f^*(\frac{ {l}_j-\mu}{\lambda}) dm\\
 = \mathbb{E}_{m} h_j.
\end{array} \right.
\end{equation}
where $h_j$ is the integrand cost $ \lambda(\rho+f(1))+\mu + \lambda f^*(\frac{ {l}_j-\mu}{\lambda}).$
We have converted the infinite dimensional problem $(P_j)$ into a finite dimensional problem $(P_j^*).$ The above calculations culminate in the following result: 
\begin{proposition}
If $a, \lambda^*(a), \mu^*(a),$ is a solution of $(P_j^*)$ then
the optimal likelihood $L^*$ is such that $\int_{\omega} L^* dm=1,\ f'(L^*)=\frac{ l_j-\mu^*}{\lambda^*}.$ This means that
$a_j$ and  $ d\tilde{m}^*=L^*dm$ provide a  solution of the original problem   $(P_j).$
\end{proposition}
\begin{proof}
Let  $ \lambda^*(a), \mu^*(a)$ be solution to  $(P_j^*)$ associated with the profile $a.$ Then, the optimal likelihood $L^*$ is obtained by differentiating $f^*$ or by inverting the equation $f'(L^*)=\frac{ l_j-\mu^*}{\lambda^*}.$  As $\tilde{m}$ is a probability measure, and using the definition of $L^*,$ one gets
$$
d\tilde{m}^*(\omega)=  L^*dm(\omega).
$$
It follows  that $a^*_j, L^*$ solves the original problem    $(P_j).$
\end{proof}

Next we look at the existence of robust equilibria.
\subsubsection{Existence of distributionally robust equilibria}

As in classical game theory, sufficiency condition for existence of robust equilibrium can be obtained from the standard fixed-point theory which we recall next. 
Let $\mathcal{A}_j$ be nonempty compact convex sets and $l_j^*$  be continuous functions such 
that for any fixed  $(z_k)_{k\neq j}, $ the function $z_j  \mapsto l_j^*(z,m)$  is quasi-convex for each $j.$ Then, there exists at least one distributionally   robust {\it pure} equilibrium.

This result can be easily extended to the coupled-action constraint case for robust generalized Nash equilibria.
\subsubsection{Performance Evaluation} 
Using 
$$
\begin{array}{ll}
l_j^*(a,\lambda,\mu,m)=&
\lambda(\rho+f(1))+\mu \\  &+\lim_{N_j \rightarrow +\infty}\frac{1}{N_j}\sum_{k=1}^{N_j} \lambda f^*(\frac{ {l}_j(., \omega_k)-\mu}{\lambda}) 
\end{array} 
$$
where $\omega_k\sim m.$
Let $$m_{N_j}=\frac{1}{N_j}\sum_{k=1}^{N_j} \delta_{\omega_k}$$ be the empirical measure of the channel state and define 
$$
\epsilon_{N_j}\sqrt{N_j} =$$ $$\sqrt{N_j}\sup_{\tilde{m}\in D_{\rho_{N_j}}(m_{N_j})} \mathbb{E}_{\tilde{m}} l_j -\sqrt{N_j} \mathbb{E}_{m_{N_j}} l_j -\sqrt{N_j \rho_{N_j} \mbox{var}_{m_{N_j}}[l_j]}
$$
with  $N_j \rho_{N_j} <+\infty.$ 
Then, the following  holds:
$\epsilon_{N_j}\sqrt{N_j} \rightarrow 0$ as $N_j$ grows.
The above result states that the robust performance captures the risk by considering the variance and not just  the ergodic performance. 


\section{Bregman Learning Algorithms} \label{sec:three}
In this section we develop learning algorithms for  $({P}^*_j)_j$.

\subsection{Maximum Principle Features}
Consider the optimal control problem $\ \inf_{u\in U} \int_0^ T \hat{l}(t,z,u) dt$ such that $\dot{z}=u.$  The maximum principle is a necessary condition of 
optimality when the underlying function is sufficiently smooth. The adjoint variable $\dot{p}=-H_z=- \hat{l}_z$ and the optimal control optimizes the 
Hamiltonian $H(z,p)=\inf_{u\in U}\{  \hat{l}+ p u\} $ i.e., the Legendre-Fenchel transform of $- \hat{l}$ applied at the point $-p.$ A closed-form expression of the 
optimal control can be obtained and it is generically given by $u^*=H_p(z,p).$ A necessary condition for optimality  says that 
$H_{u^*} (u-u^*)\geq 0$ for any $u\in U,$ where $H_u$ denotes an element of the sub-differential of $H.$ This latter variational equation can be rewritten as
\begin{equation}
\begin{array}{ll}
0\leq H_{u^*} (u-u^*)=[\hat{l}_{u^*}+p](u-u^*),\ 
\end{array}
\end{equation}
for all $u\in U.$

In particular, an interior solution $u^*$ should solve $p=-\hat{l}_{u^*}$ and the adjoint equation becomes 
$
\dot{p}=\frac{d}{dt}(-\hat{l}_{u^*})=-\hat{l}_z(z,u^*),
$
which means that $$\frac{d}{dt}\hat{l}_{\dot{z}}=\hat{l}_z(z,\dot{z}).$$ The latter equation is also called Euler-Lagrange equation in the field of calculus of variations.
Since the minimization is among all possible curves, this minimum principle may exhibit features that allow to investigate faster time curves.

\subsection{Bregman Speedup Learning}
Let $g: \ \mathcal{A} \rightarrow \mathbb{R} $ be a differentiable, strictly convex function.
The Bregman divergence \cite{bregman} is $d_{g}: \mathcal{A} \times relint(\mathcal{A} ) \rightarrow \mathbb{R}$ and is defined as
$$d_g^{BR}(y,x)=g(y)-g(x)-\langle g_x(x), y-x\rangle,$$ where $relint(\mathcal{A} )$ denotes the relative interior of $\mathcal{A} .$ 

We investigate the equation $\frac{d}{dt}\hat{l}_{u}(z,u)=\hat{l}_z(z,u)$ for a class of quantity-of-interest $\hat{l}.$
Let the family of Bregman-based Lagrangian be $ \hat{l}(z,u)=e^{\alpha+\gamma}[d_{g}(z+e^{-\alpha}u,z)-e^{\beta} l^*(z)].$

\begin{proposition} \label{lem1}
The Euler-Lagrange equation reduces to the following second order  differential system, for $\dot{\gamma}=e^{\alpha},$
\begin{equation}
\begin{array}{ll}
 \ddot{z}+(e^{\alpha}-\dot{\alpha})\dot{z}+e^{2\alpha+\beta}g_{zz}^{-1}(z+e^{-\alpha}\dot{x}) l^*_z(z)
=0.
\end{array}
\end{equation} 
\end{proposition}
\begin{proof} 
We start with the definition of Bregman divergence. 
A simple computation shows that
$$
\partial_y d_g(y,x)=g_x(y)-g_x(x),\ \  \partial_x d_g(y,x)=- g_{xx}(x).(y-x).
$$
By differentiating the functional $\hat{l}$ one gets
\begin{equation}
\begin{array}{ll}
 \hat{l}_z= e^{\alpha+\gamma}[-d_{g,1}(z+e^{-\alpha}u,z)-d_{g,2}(z+e^{-\alpha}u,z)
+e^{\beta} l^*_z]\\ =e^{\alpha+\gamma}[- g_z(y)+g_z(z)+g_{zz}(z)(y-z)+e^{\beta} l^*_z], \\
 \hat{l}_{u}=-e^{\gamma} d_{g,1}(z+e^{-\alpha}u,z)  = -e^{\gamma}[g_z(z+e^{-\alpha}u)-g_z(z)]
\end{array}
\end{equation}

It follows that 
\begin{equation}
\begin{array}{ll}
\frac{d}{dt}\hat{l}_u= -\dot{\gamma}e^{\gamma} [g_z(z+e^{-\alpha}u)-g_z(z)]\\ -e^{\gamma} \frac{d}{dt}[g_z(z+e^{-\alpha}u)-g_z(z)]\\
=-\dot{\gamma}e^{\gamma} [g_z(z+e^{-\alpha}u)-g_z(z)]\\ -e^{\gamma} [ g_{zz}(z+e^{-\alpha}u)(\dot{z}-\dot{\alpha}e^{-\alpha}\dot{z}+e^{-\alpha}\ddot{z})-g_{zz}(z)\dot{z}]
\\ =\hat{l}_z\\
=e^{\alpha+\gamma}[- g_z(y)+g_z(z)+g_{zz}(z)(y-z)+e^{\beta} l^*_z].
\end{array}
\end{equation}
Thus,
\begin{equation}
\begin{array}{ll}
-\dot{\gamma} [g_z(z+e^{-\alpha}\dot{z})-g_z(z)] \\ -[ g_{zz}(z+e^{-\alpha}\dot{z})(\dot{z}-\dot{\alpha}e^{-\alpha}\dot{z}+e^{-\alpha}\ddot{z})-g_{zz}(z)\dot{z}]
\\ 
=e^{\alpha}[- g_z(z+e^{-\alpha}\dot{z})+g_z(z)+g_{zz}(z)e^{-\alpha}\dot{z}+e^{\beta} l^*_z].
\end{array}
\end{equation}
By rearranging the terms in $\ddot{z}, \dot{z}$ we obtain
\begin{equation}
\begin{array}{ll}
e^{\alpha}(\dot{\gamma}-e^{\alpha})g_{zz}^{-1}(z+e^{-\alpha}\dot{z}) [g_z(z+e^{-\alpha}\dot{z})-g_z(z)] \\ +e^{\alpha} g_{zz}^{-1}(z+e^{-\alpha}\dot{z})[  g_{zz}(z+e^{-\alpha}\dot{z})(\dot{z}-\dot{\alpha}e^{-\alpha}\dot{z})
-g_{zz}(z)\dot{z}]
\\ 
+e^{2\alpha}g_{zz}^{-1}(z+e^{-\alpha}\dot{z})[g_{zz}e^{-\alpha}\dot{z}+e^{\beta} l^*_z]
+\ddot{z}
=0.
\end{array}
\end{equation}
Taking $\dot{\gamma}=e^{\alpha}$ yields
\begin{equation}
\begin{array}{ll}
e^{\alpha} g_{zz}^{-1}(z+e^{-\alpha}\dot{z})[  g_{zz}(z+e^{-\alpha}\dot{z})(\dot{z}-\dot{\alpha}e^{-\alpha}\dot{z})]
\\ 
+e^{2\alpha}g_{zz}^{-1}(z+e^{-\alpha}\dot{z})[g_{zz}e^{-\alpha}\dot{z}+e^{\beta} l^*_z]
+\ddot{z}
=0.
\end{array}
\end{equation}
The Euler-Lagrange equation reduces to
\begin{equation}
\begin{array}{ll}
 \ddot{z}+(e^{\alpha}-\dot{\alpha})\dot{z}+e^{2\alpha+\beta}g_{zz}^{-1}(z+e^{-\alpha}\dot{z}) l^*_z(z)
=0,
\end{array}
\end{equation}
which can be rewritten as
\begin{equation}
\begin{array}{ll}
 [\ddot{z}e^{-\alpha}+(1-\dot{a}e^{-\alpha})\dot{z}]g_{zz}(z+e^{-\alpha}\dot{z})+e^{\alpha+\beta} l^*_z(z)=0.
\end{array}
\end{equation}
From the above, the Bregman algorithm yields
\begin{equation}
\begin{array}{ll}
\frac{d}{dt}[ g_z(z+e^{-\alpha}\dot{z})]=-e^{\alpha+\beta} l^*_z(z).
\end{array}
\end{equation} This completes the proof.
\end{proof}
Note that the second order system is easily converted into a first-order system by setting 
$$\left\{ \begin{array}{l}
y=z+e^{-\alpha}u,\\
\dot{z}=u,\\
\dot{u}=\ddot{z}= - (e^{\alpha}-\dot{\alpha})u-e^{2\alpha+\beta}g_{zz}^{-1}(z+e^{-\alpha}u) l^*_z(z).
\end{array}\right.
$$
\begin{defi}
We say that $z\mapsto \tilde{l}^*(z)$ is a best response pseudo-potential function for the distributionally robust game $\mathcal{G}(m)$ if 
$$\arg\min_{z_j}  \tilde{l}^*(z)\subseteq \arg\min_{z_j}  {l}^*_j(z), \quad\mbox{for every $j.$}$$ 
\end{defi}
The Bregman algorithm is given by 
\begin{equation} \label{breg}
\begin{array}{ll}
\frac{d}{dt}[ g_z(z+e^{-\alpha}\dot{z})]=-e^{\alpha+\beta} \tilde{l}^*_z(z),\quad 
z(0)=z_0,
\end{array}
\end{equation}
where $\beta(t)=\beta(0)+\int_0^t e^{\alpha(t')}\ dt',$  $\beta(0)\geq 0,$ and $\alpha$ is a time-dependent function.
 \begin{proposition} \label{thm2}
 If $ \tilde{l}^*$ is convex then 
 $$0\leq  -\tilde{l}(z^*)+ \tilde{l}(z(t))\leq e^{-\beta(t)} c_0,$$
 where $$
 c_0=d_g^{BR}(z^*,z_0+e^{-\alpha(0)}\dot{z}_0)+e^{\beta(0)}[-\tilde{l}(z^*)+\tilde{l}(z(0))].
 $$
 \end{proposition}
By choosing $\alpha(t)=t,$ $\beta(t)=e^t,$ and the error gap is $$
 -\tilde{l}(z^*)+ \tilde{l}(z(t))\leq e^{-e^t} c_0. $$
It takes $T_{\eta}=\log \log (\frac{c_0}{\eta})$ time units to reach a neighborhood of the equilibrium payoff of $z^*$ with a precision $\eta>0.$ This is faster than 
 Ishikawa-Nesterov  algorithm $O(\frac{1}{\sqrt{\eta}})$, gradient ascent method $O(\frac{1}{\eta}),$  no-regret dynamics, and back-box optimization $\frac{c_0}{\eta^2}.$ Thus, Bregman dynamics speeds up the learning and  improves classical methods with exponential decay.
 
 The proof of Proposition \ref{thm2} is based on a careful construction of a generalized best-response pseudo-potential function using Pontryagin maximum principle. It extends the framework developed in \cite{jordan1} to the context of strategic-form games. 
 Then we check that  the following function $V$ is a Lyapuvov function:  $$V(z^*,z(t))=d_g^{BR}(z^*,z(t)+e^{-\alpha(t)}\dot{z}(t))+e^{\beta(t)}[-\tilde{l}(z^*)+\tilde{l}(z(t))]$$ where $z(t)$ is generated by the Bregman algorithm. 
 
 Note that Proposition \ref{thm2} does not require the strong convexity property often used in the proof of convergence  gradient dynamics and Newton-based gradient methods. This is because the Bregman divergence is carefully designed to compensate that part. Table \ref{tab} summarizes the theoretical speedup advantages of Bregman algorithms over the state-of-the-art algorithms.
  \begin{table}[htb] 
\begin{tabular}{|l | l| l | l|}\hline
 &  Accuracy &  Time   to Reach \\ \hline
 This paper&  $O(e^{-\beta(t)})$ &$O(\beta^{-1}(\log (\frac{1}{\eta})))$  \\ \hline
  This paper (Bregman) &  $O(e^{-e^t})$ &$O(\log (\log (\frac{1}{\eta})))$  \\ \hline
Ishikawa-Nesterov&  $O(\frac{1}{t^2})$ &$O(\frac{1}{\sqrt{\eta}})$  \\ \hline
 Conjugate/proximal gradient  &   $O(\frac{1}{t})$ &  $O(\frac{1}{\eta})$ \\ \hline
 Gradient ascent &  $O(\frac{1}{t})$&  $O(\frac{1}{\eta})$  \\ \hline
Regret-min &  $O(\frac{\log t}{t})$ & -  \\  \hline
Standard black-box &  $O(\frac{1}{\sqrt{t}}+.)$ & $O(\frac{1}{\eta^2})$   \\ \hline
\end{tabular}\\  \caption{Performance of the proposed Bregman algorithm compared to the classical ones with a precision error within $\eta>0.$ } \label{tab}
\end{table}
\begin{proof}[Proof of Proposition \ref{thm2}]
Let us define function $V$ as follows.
\begin{equation}
\begin{array}{ll}
V(z,u,t, z^*)= d_g(z^*, z+e^{-\alpha}u)+e^{\beta}[-l^*(z^*)+l^*(z) ].
\end{array}
\end{equation}
Then  the function $V(z,u,t)$  is positive.
 Let us compute the time derivative of $V$ over the path $z(t),u(t)$ generated by the Bregman algorithm.
$$\frac{d}{dt}[ g_z(z+e^{-\alpha}\dot{z})]=-e^{\alpha+\beta} l^*_z(z).$$
We also have that \begin{equation}
\begin{array}{ll}
\frac{d}{dt}V(z(t),u(t),t, z^*)=\\
-\frac{d}{dt}[z+e^{-\alpha}u] g_{zz}(z+e^{-\alpha}u).(z^*-z-e^{-\alpha}u)\\ +\dot{\beta}e^{\beta}[l^*(z)-l^*(z^*)]+e^{\beta} l^*_z(z) \dot{z}
\end{array}
\end{equation}
By summing and subtracting the same term we have
\begin{equation}
\begin{array}{ll}
\frac{d}{dt}V(z(t),u(t),t, z^*)= \\
-\frac{d}{dt}[z+e^{-\alpha}u] g_{zz}(z+e^{-\alpha}u).(z^*-z-e^{-\alpha}u)\\ +\dot{\beta}e^{\beta}[l^*(z)-l^*(z^*)+l^*_{z}(z^*-z)-l^*_{z}(z^*-z)]\\ +e^{\beta} l^*_z(z) \dot{z} 
=-\frac{d}{dt}[g_{z}(z+e^{-\alpha}u)]. (z^*-z-e^{-\alpha}u)\\ +\dot{\beta}e^{\beta}[ l^*(z)-l^*(z^*)+l^*_{z}(z^*-z)]
\\
-\dot{\beta}e^{\beta} l^*_{z}(z^*-z)+e^{\beta} l^*_z(z) \dot{z}.
\end{array}
\end{equation}
The above leads to further expansion as follows
\begin{equation}
\begin{array}{ll}
\frac{d}{dt}V(z(t),u(t),t, z^*)=e^{\alpha+\beta} l^*_z(z)(z^*-z-e^{-\alpha}u)+\\  \dot{\beta}e^{\beta}[ l^*(z)-l^*(z^*)+l^*_{z}(z^*-z)] \\
-\dot{\beta}e^{\beta} l^*_{z}(z^*-z)+e^{\beta} l^*_z(z) \dot{z}\\
=e^{\alpha+\beta} l^*_z(z)(z^*-z)+e^{\beta}  l^*_z(z)(\dot{z}-u) \\
-\dot{\beta}e^{\beta} l^*_{z}(z^*-z)
+ \dot{\beta}e^{\beta}[l^*(z)-l^*(z^*)-l^*_{z}(z-z^*)]\\
= e^{\beta}( e^{\alpha}-\dot{\beta}) l^*_{z}(z^*-z)\\
+ \dot{\beta}e^{\beta}[ l^*(z)-l^*(z^*)-l^*_{z}(z-z^*)].
\end{array}
\end{equation}
By convexity of the function $l^*, $  $ [l^*(z)-l^*(z^*)-l^*_{z}(z-z^*)]\leq 0$ and $l^*_{z}(z^*-z)\leq 0.$
If $e^{\alpha}-\dot{\beta}\geq 0$ then
 \begin{equation}
\begin{array}{ll}
\frac{d}{dt}V(z(t),u(t),t, z^*)=e^{\beta}( e^{\alpha}-\dot{\beta})l^*_{z}(z^*-z)\\ + \dot{\beta}e^{\beta}[l^*(z)-l^*(z^*)-l^*_{z}(z-z^*)]\leq 0.
\end{array}
\end{equation}

Thus, $\frac{d}{dt}V(z(t),u(t),t, z^*)\leq 0$ for  $\dot{\beta}\leq e^{\alpha}.$ Then the function is decreasing  over the path of the Bregman algorithm. It follows that
$ e^{\beta}[l^*(z)-l^*(z^*)]\leq V(z,u,t, z^*)\leq V(z_0,u_0,t_0).$
Then, the global error is 
$$0\leq l^*(z)-l^*(z^*)\leq  e^{-\beta} V(z_0,u_0,t_0),$$
with $\dot{\beta}\leq e^{\alpha},$ which shows an exponential convergence to $z^*.$ This completes the proof.
\end{proof}

 \section{Stochastic Bregman Learning} 
 \label{sec:four}
 
  Very often the computation of the terms $\nabla \mathbb{E} l^*,$  $\ \mathbb{E}[ \lambda(\rho+f(1))+\mu+f^*(\frac{l_j-\mu}{\lambda})]$ or its partial derivatives is  challenging and depends on the structure of the distribution $m.$  We now propose a swarm learning to estimate the expected gradient and then insert it to the Bregman algorithm (\ref{breg}), leading to  particle swarm stochastic  Bregman algorithm.
 
 \subsection{Single particle }
We propose a stochastic Bregman learning framework which is adjusted based only 
on the realized integrand $h_j(z,\omega_j):=\lambda(\rho+f(1))+\mu+\lambda f^*(\frac{l_j-\mu}{\lambda}).$ The expected value of $h_j$ is  $\mathbb{E}_{\omega \sim m}h_j  =l_j^*.$
The stochastic Bregman dynamics yields
$$ \begin{array}{l}
\frac{d}{dt}[g_{j,z_j}(z+e^{-\alpha}u)]=-e^{\alpha+\beta} h_{j,z_j}(z,\omega_j)\\ =-e^{\alpha+\beta} [\mathbb{E}h_{j,z_j}(z,\omega_j)+ W_j]= -e^{\alpha+\beta} [l^*_{j,z}(z)+W_j(z,\omega_j)],\\
i\in \{1,2,\ldots, n\}.
\end{array}
$$
Function  $z=(a,\lambda,\mu)$ is now a stochastic process because of the stochastic term $\omega_j$ and $W_j(z)=h_{j,z_j}(z,\omega_j)-\mathbb{E}h_{j,z_j}(z,\omega_j).$ The variance of $W$ is being high and not vanishing because it is based on a single particle path discrepancy. We introduce in the following subsection a swarm of particles. 
 \subsection{Swarm of particles }
 Let us associate to each agent  $j$  a swarm of virtual particles $\omega_{jk}.$ Then, we have 
$$ \mathbb{E}h_{j,z_j}(z,.)=\lim_N \frac{1}{N}\sum_{k=1}^{N} h_{j,z_j}(z,\omega_{jk}). $$
Swarm-based stochastic Bregman dynamics yields
\begin{equation} \label{eqn2}\begin{array}{l}
\frac{d}{dt}[g_{z_j}(z+e^{-\alpha}u)]=-e^{\alpha+\beta}  \frac{1}{N}\sum_{k=1}^{N} h_{j,z_j}(z,\omega_{jk}),\\
\omega_{jk}\sim m, \quad 
j\in \{ 1,2,\ldots, n \}.
\end{array}
\end{equation} 
This is a mean-field-type interacting system and  can be seen  as control-dependent correlated noise modification of the Bregman dynamics as
$$ \begin{array}{l}
\frac{d}{dt}[g_{z_j}(z+e^{-\alpha}u)]=-e^{\alpha+\beta} [l^*_{j,z_j}(z)+\epsilon_{j,N}(z,\omega)],\\
j\in \{ 1,2,\ldots, n \},
\end{array}
$$
where  $\epsilon_{j,N}=\frac{1}{N}\sum_{k=1}^{N} h_{j,z_j}(z,\omega_{jk})-l^*_{j,z_j}(z)$ has zero mean and standard deviation as $$\sqrt{\mathbb{E}[\epsilon^2_{j,N}]}=\sqrt{\frac{\mbox{var}{[h_{j,z_j}(z,.)]}}{N^2}}.$$ 
For a realized $\omega,$ set 
\begin{eqnarray*}l^*_{j,N}=\frac{1}{N}\sum_{k=1}^{N} h_{j}(z,\omega_{jk}), \quad 
z^*_{j,N}\in \arg\min_{z}\frac{1}{N}\sum_{k=1}^{N} h_{j}(z,\omega_{jk}).\end{eqnarray*} Then, the particle swarm Bregman algorithm (\ref{eqn2}) gives as output the function $z_N(t)$ that satisfies
$$0\leq  -\tilde{l}_N(z^*_N)+ \tilde{l}_N(z_N(t))\leq e^{-\beta(t)} c_{0,N}$$
where $c_{0,N}:=d_g^{BR}(z^*_N,z_0+e^{-\alpha(0)}\dot{z}_0)+e^{\beta(0)}[-\tilde{l}_N(z^*_N)+\tilde{l}_N(z_0)].$

This says that the $N-$swarm per player Bregman scheme provides a good approximation of the robust equilibrium.

 \section{Numerical Investigation}\label{sec:five}
 To illustrate the particle swarm Bregman algorithm (\ref{eqn2})  we consider  specific robust games. 
 We consider two agents and the discriminator/adversary.  We choose 
 \begin{equation}
  f(x)=\left\{ \begin{array}{ll}
x\log x- x & \mbox{if}\ x>0,\\
0 & \mbox{if}\ x=0.
 \end{array} \right.
 \end{equation}
 We compute  $f(1)=-1$ and  $f'(x)=\log x,$ 
 $f''(x)=\frac{1}{x}>0.$ Hence $f$ is convex on $\mathbb{R}_+.$ The Legendre-Fenchel transform of $f$ yields
 $$
 f^*(\xi)=e^{\xi}.
 $$
 
  \subsection{Best-response Pseudo-Potential  Distributionally Robust Game}
  
  We set 
 $$l_j(a,\omega)= \log\left(1+\omega_1^2 a_1^2+\omega_2^2 a^2_2\right)$$ defined over $\mathbb{R}^2.$ The integrand function $h_j$ 
 is
 $$
 h_j=\lambda (\rho-1)+\mu +\lambda (1+\omega_1^2 a_1^2+\omega_2^2 a^2_2)^{\frac{1}{\lambda}} e^{-\frac{\mu}{\lambda}}
 $$
  The random variable $\omega$ is distributed over $m$
  and we assume that  $\omega$ has finite moments. 
  The stochastic robust objective function $l^*_{j,N}:$
  $$
 l^*_{j,N}=\lambda (\rho-1)+\mu + \frac{\lambda}{N}\sum_{k=1}^N (1+\omega_{1k}^2 a_1^2+\omega_{2k}^2 a^2_2)^{\frac{1}{\lambda}} e^{-\frac{\mu}{\lambda}}
 $$

 \begin{figure*}[htb]
\includegraphics[scale=0.3]{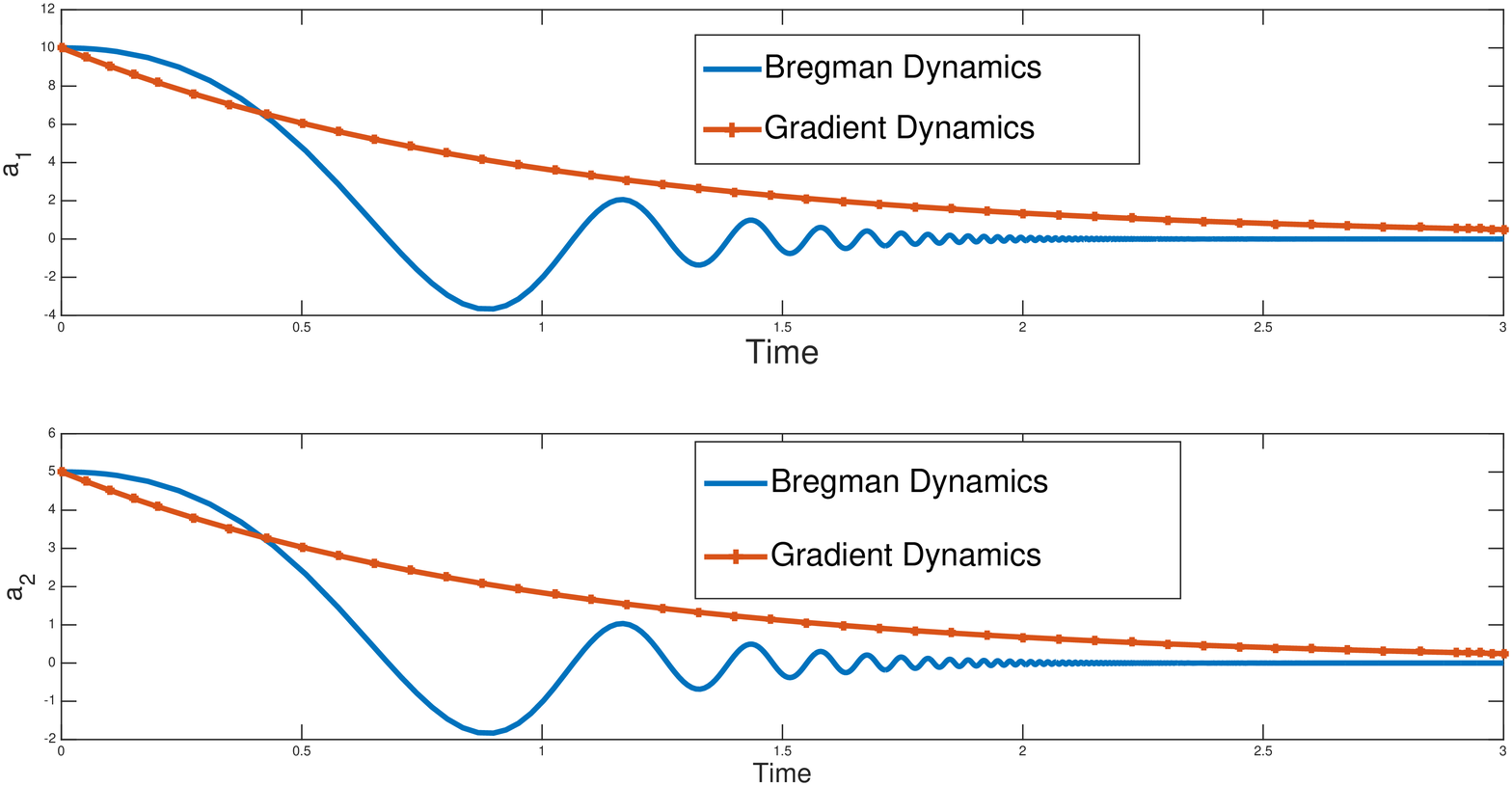}\\
\caption{Gradient vs Bregman-based dynamics.} \label{fig00v1}
\end{figure*}
  
 We illustrate the   Bregman-based dynamics in Figure \ref{fig00v1} $N=1000$ samples. We observe a rapid convergence to the
 robust equilibrium. The trajectory is not a descent but the amplitude of oscillation quickly decreases and an acceptable convergence time that is 20 times better  than the classical gradient dynamics.

 \subsection{Non-Convex Setting}
 
  We set 
 $$l_j(a,\omega)= \log\left(20 +\omega^2- \sin(a_1)\sin(a_2)\sqrt{a_1a_2}\right),$$ over the action space $[0, 10]^2.$ 
 The function $l_j$ has multiple local extrema as illustrated in Figures \ref{fig:one} and \ref{fig:onefrench}. The objective function of agent $j$ is non-convex and non-concave. The function $l_j$ is chosen because it does not fulfill the conditions of the Theorem \ref{thm2}. We observed that the Bregman algorithm behaves well even for this multimodal case which opens the investigation for non-convex objective functions. The multimodal function has a robust equilibrium around $(7.9, 7.9) $ and the equilibrium performance is around $ 7.88.$ In Figure \ref{fig:fit}, the  Bregman learning   outcome changes from the distributionally robust  Nash equilibrium $(7.9, 7.9)$ to a local maximum  $(7.9, 5.1).$
\begin{figure}[htb]
\includegraphics[width=7cm,height=7cm]{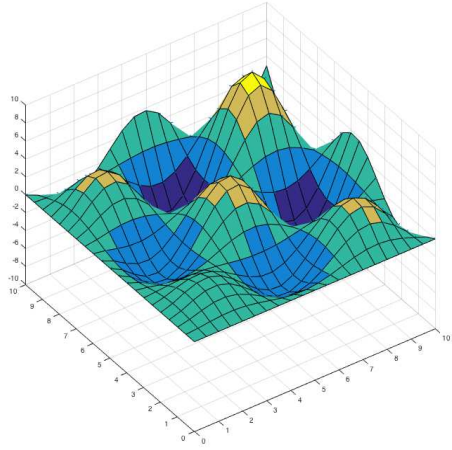}
\caption{The function $l_j$ is non-convex,  non-concave and has multiple local extrema.}
\label{fig:one}
\end{figure}

\begin{figure}[htb]
\includegraphics[width=7cm,height=7cm]{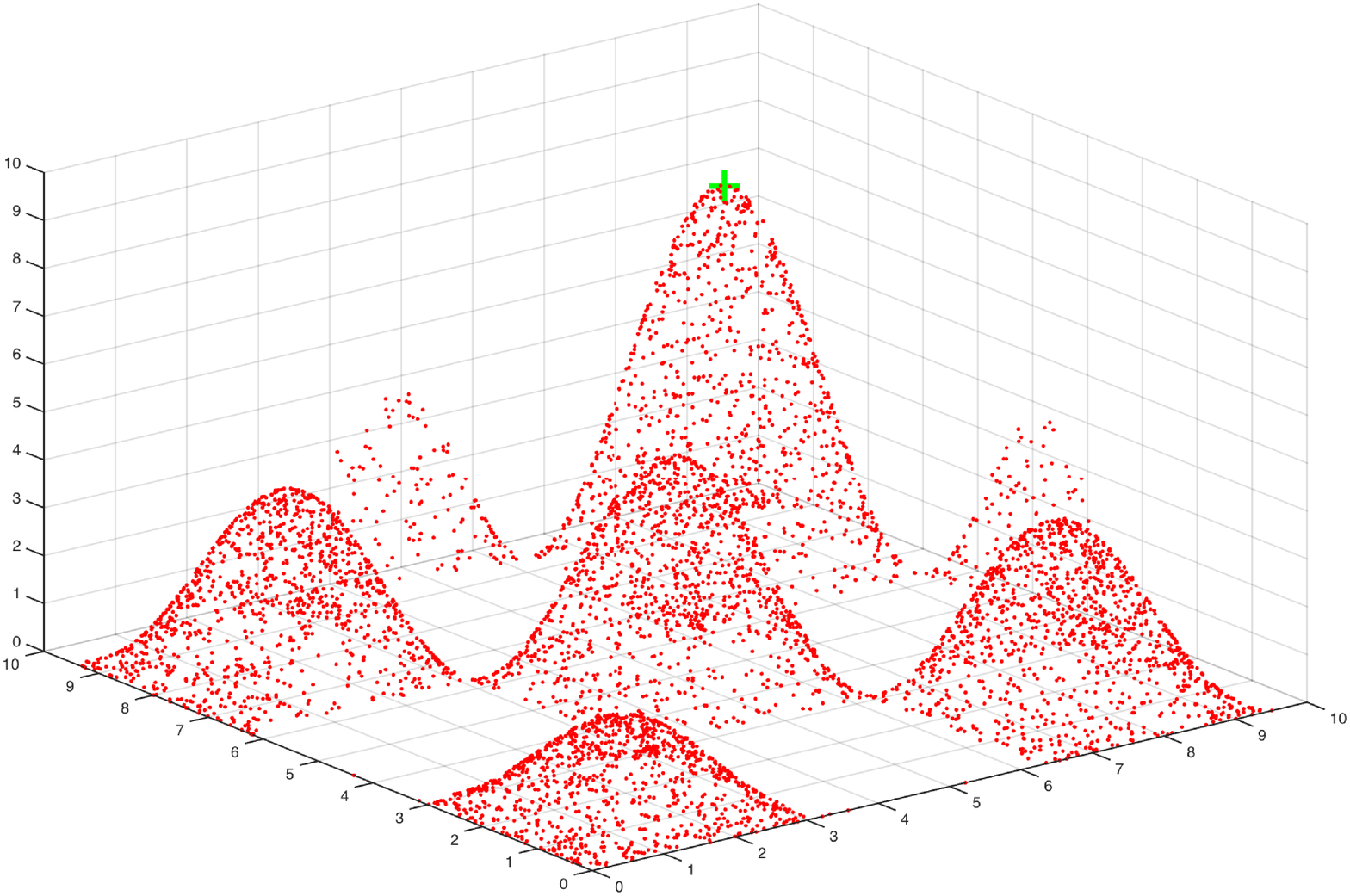}
\caption{  The particle swarm Bregman learning   leads to the robust  Nash equilibria  around $(7.9, 7.9) $ and the equilibrium value is around $ 7.88.$}
\label{fig:onefrench}
\end{figure}
 
\begin{figure}[htb]
\includegraphics[width=9cm,height=12cm]{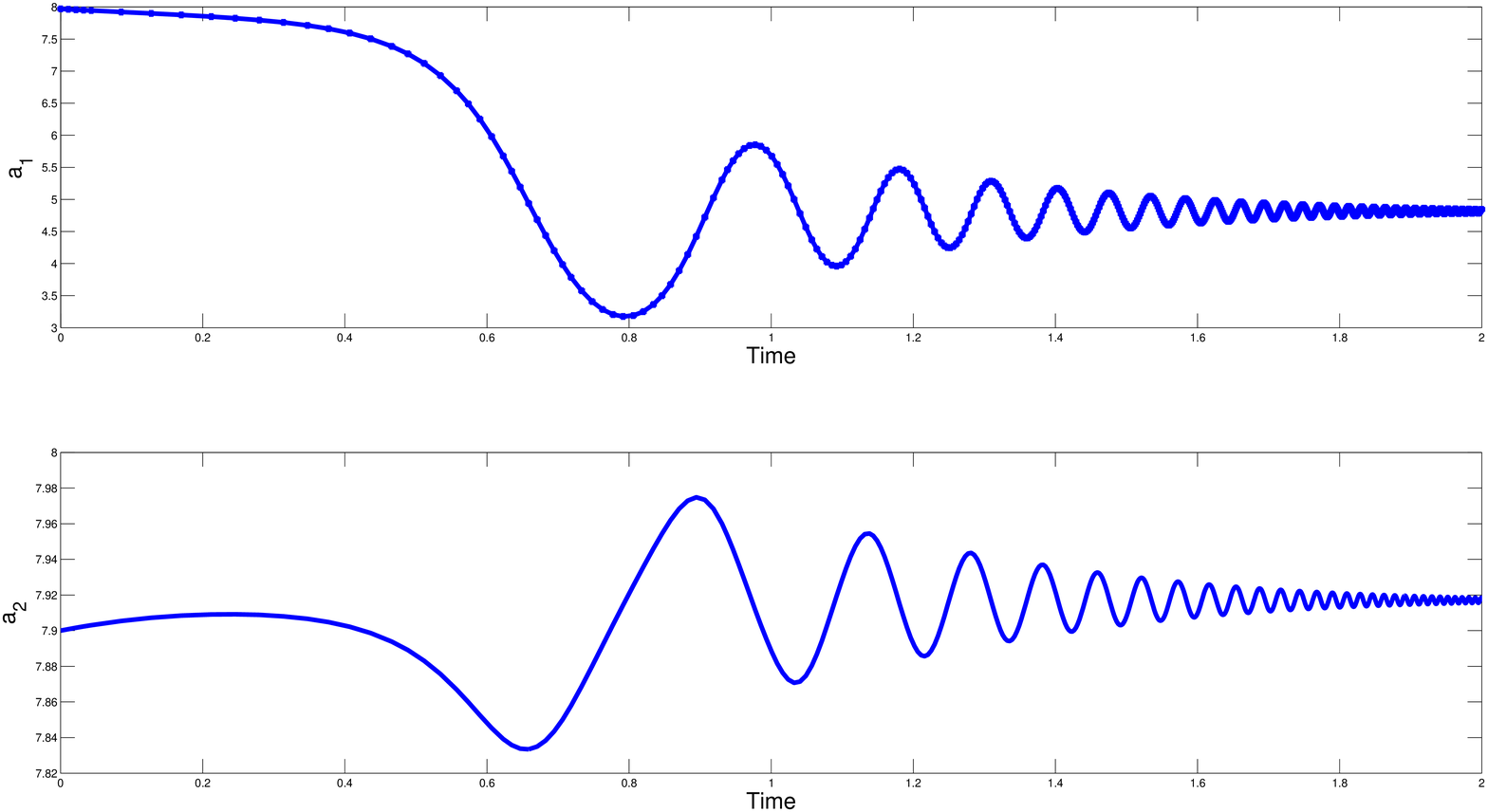}
\caption{  The  Bregman learning   outcome changes from the robust  Nash equilibrium $(7.9, 7.9)$ to a local maximum  $(7.9, 5.1).$ }
\label{fig:fit}
\end{figure}

%
%

\section{Finite action spaces case} \label{sec:six}
We now discuss how the above framework can be used in the discrete action space case. We limit our exposition to the finite set. $\mathcal{A}_j$ is a finite (with two or more actions), it is not a convex set. We convexify them as standard mixed strategy approach. $\mathcal{A}_j$ will be replaced by $\mathcal{X}_j$ the simplex over $\mathcal{A}_j.$ The robust payoff on pure action profile will be replaced by the expected robust payoff.  One obtains the so-called {\it mixed extension} of the game. The existence of distributionally robust mixed equilibrium follows from standard fixed-point theorems.

\section*{Acknowledgment}
This research is supported by U.S. Air Force Office of Scientific Research under grant number FA9550-17-1-0259. This research was conducted when D. Bauso was visiting NYUAD.

\section{Conclusion}
\label{sec:last} 
We have introduced distributionally robust game with continuous action space for each agent and a possible adversarial modification of the uncertainty. The problem is formulated using a notion of divergence between two measures: the modified measure  and the exact measure associated to the uncertainty.
 In the context of existence of robust solutions, additional difficulties arise if in addition a robustness condition or an adversarial control of the distribution is involved in the objective function. We have used  triality theory to transform the objective function of each agent. This transformation reduces considerably the curse of dimensionality of the problem. Then,  sufficient conditions of existence of solution are derived.
We constructed a speedup learning algorithm based on Bregman discrepancy. The methodology does not require strong convexity assumptions as in the classical gradient algorithms. The convergence time is shown to be much faster than the current state-of-the-art algorithms developed for pseudo-potential games. Our future work aims to use and apply  the approach to adversarial generative networks.

\bibliographystyle{ACM-Reference-Format}

\begin{thebibliography}{99}
\bibitem{free}M. Aghassi and D. Bertsimas. Robust game theory. Mathematical Programming, 107(1-2):231-273, 2006.
\bibitem{free2}Dimitris Bertsimas, David B Brown, and Constantine Caramanis. Theory and applications of
robust optimization. SIAM review, 53(3):464-501, 2011.

\bibitem{conti6}
Barry D. Nichols: A Comparison of Action Selection Methods for Implicit Policy Method Reinforcement Learning in
Continuous Action-Space, IEEE World Congress on Computational Intelligence (IEEE WCCI), Vancouver, Canada, 24-29 July 2016.

\bibitem{jordan1}
Wibisono, Andre and  Wilson, Ashia C. and  Jordan, Michael I.,
A variational perspective on accelerated methods in
optimization, In proceedings of the National Academy of
Sciences, 113, {47},  2016. 

\bibitem{conti5}D. Friedman, D. N. Ostrov: Gradient dynamics in population games: Some basic results, Journal of Mathematical Economics,
vol. 46, Issue 5, 20 Sept. 2010, pp. 691-707, Mathematical Economics: Special Issue in honour of Andreu Mas-Colell, Part 1
\bibitem{conti4} D. Friedman, D. N. Ostrov: Evolutionary dynamics over continuous action spaces for population games that arise from symmetric two-player games, 
Journal of Economic Theory, vol. 148, Issue 2, March 2013, pp. 743-777
\bibitem{conti3}
S. Perkins and D. S. Leslie, Stochastic fictitious play with continuous action sets, J. Econ. Theory, vol. 152, pp. 179-213, 2014.
\bibitem{conti2}
M.W. Cheung, Pairwise comparison dynamics for games with continuous strategy space, J. Econ. Theory, vol. 153, pp. 344-375, 2014.
 
\bibitem{conti1}
L. Pavel, Games with Continuous Action Spaces, In 
Game Theory for Control of Optical Networks,
Part of the series Static \& Dynamic Game Theory: Foundations \& Applications, pp 45-54, ISBN
978-0-8176-8321-4, 2012

\bibitem{bregman}
L. M. Bregman. The relaxation method of finding the common point of convex sets and its application to the solution of problems in convex programming. USSR Computational Mathematics and Physics, 7, 200-217, 1967.
\bibitem{sion} Sion, M. (1958). On general minimax theorems. Pacific Journal of Mathematics. 8 (1): 171-176
\bibitem{dga} 
D. Bauso, H. Tembine, T. Basar,  Robust Mean-Field Games, Journal of Dynamic Games and Applications, 27 pages, DOI 10.1007/s13235-015-0160-4, 
September 2016, volume 6, Issue 3, pp 277-303.
\bibitem{book}
H. Tembine, Distributed strategic learning for wireless engineers, 2012, CRC Press, 496 Pages,ISBN 9781439876374
\end{thebibliography}

\end{document}